\documentclass{amsart}

\usepackage[all]{xy}

\theoremstyle{plain}
\newtheorem{thm}{Theorem}

\newtheorem{cor}[thm]{Corollary}

\theoremstyle{definition}

\newtheorem{hypo}[thm]{Hypothesis}

\theoremstyle{remark}

\begin{document}

\title{Dihedral blocks with two simple modules}

\author{Frauke M. Bleher}
\address{Department of Mathematics\\University of Iowa\\
Iowa City, IA 52242-1419, U.S.A.}
\email{fbleher@math.uiowa.edu}
\thanks{The author was supported in part by NSF Grant DMS06-51332.}

\date{January 5, 2010}

\subjclass[2000]{Primary 20C05; Secondary 16G20}
\keywords{Dihedral defect groups, generalized quaternion defect groups,
projective general linear groups}

\begin{abstract}
Let $k$ be an algebraically closed field of characteristic $2$, and let $G$ be a finite group.
Suppose $B$ is a block of $kG$ with dihedral defect groups such that there are precisely two
isomorphism classes of simple $B$-modules. The description by Erdmann
of the quiver and relations of the basic algebra of $B$ is usually only given up to a certain 
parameter $c$ which is either $0$ or $1$. 
In this article, we show that $c=0$ if there exists a central extension $\hat{G}$ of $G$
by a group of order $2$ together with
a block $\hat{B}$ of $k\hat{G}$ with generalized quaternion defect groups
such that $B$ is contained in the image of $\hat{B}$ under the natural surjection from $k\hat{G}$ onto $kG$. 
As a special case, we obtain that $c=0$ if $G=\mathrm{PGL}_2(\mathbb{F}_q)$ for some odd prime
power $q$ and $B$ is the principal block of $k\,\mathrm{PGL}_2(\mathbb{F}_q)$.
\end{abstract}

\maketitle


\section{Introduction}

Let $k$ be an algebraically closed field of arbitrary characteristic,
and let $\Lambda$ be a finite dimensional $k$-algebra.
One of the fundamental problems in representation theory is to determine the indecomposable
$\Lambda$-modules up to isomorphism. Depending on the complexity of the set of isomorphism
classes of indecomposable $\Lambda$-modules, one distinguishes between three representation types:
finite, tame and wild. If $G$ is a finite group and $B$ is a block of $kG$, then $B$ has finite 
representation type if and only if the defect groups of $B$ are cyclic, whereas $B$ has tame 
representation type if and only if the characteristic of $k$ is $2$ and the defect groups of $B$ are either
dihedral, semidihedral or generalized quaternion. 
In a series of articles culminating in the monograph
\cite{erd}, Erdmann described the quivers and relations of the
basic algebras of all blocks of group algebras of tame
representation type, by introducing the larger classes of algebras
of dihedral, semidihedral and quaternion type.
For example, if $k$ has characteristic 
$2$, then the class of all $k$-algebras $\Lambda$ of dihedral type includes, up to Morita equivalence, 
all blocks of $kG$ with dihedral defect groups.

Suppose now that $k$ has characteristic $2$. Let $\mathcal{B}_{\ell,d}$ consist of all
blocks $B$ of group algebras of finite groups over $k$ with dihedral defect groups of order $2^d$
and precisely $\ell$ isomorphism classes of simple $B$-modules. In \cite[Thm. 2]{brauer2}, 
Brauer showed 
that $\ell\le 3$. Let $\Lambda$ be a $k$-algebra of dihedral type such that $\Lambda$ is Morita 
equivalent to a block $B\in\mathcal{B}_{\ell,d}$. In \cite[Chap. VI]{erd}, the possible quivers for 
$\Lambda$ have been determined; it follows from the tables in \cite[p. 294--297]{erd}
that they only depend on the decomposition matrix of $B$. If $\ell=3$
then the relations of $\Lambda$ given in \cite[p. 294--297]{erd} only depend on $d$. If 
$\ell=2$ then the relations in \cite[p. 294--297]{erd} additionally depend on a parameter $c\in\{0,1\}$, 
and it is shown  in \cite[Sect. VI.8]{erd} that the algebras 
with the same quiver, the same $d$
but different $c$
are not Morita equivalent. 
If $B$ is a block in $\mathcal{B}_{2,d}$, it is usually difficult to decide whether $c=0$ or $c=1$.
To our knowledge there are only a few cases of blocks $B\in\mathcal{B}_{2,d}$ where $c$ has been 
determined,
such as the principal $2$-modular block of the symmetric group $S_4$ (see \cite[Cor. V.2.5.1]{erd})
or the principal $2$-modular blocks of certain quotients of  the general unitary group
$\mathrm{GU}_2(\mathbb{F}_q)$, where $q\equiv 3 \mod 4$ (see \cite[Sect.1.5]{erdlater}). 

In this paper, we consider the following situation, where as before $k$ has
characteristic $2$. Let $G$ be a finite group, and suppose $B$ is a block of $kG$ 
belonging to $\mathcal{B}_{2,d}$. Suppose further that there exists a central extension $\hat{G}$ of $G$ 
by a group of order $2$ together with a block $\hat{B}$ of $k\hat{G}$ with generalized quaternion defect groups
of order $2^{d+1}$ such that $B$ is contained in the image of $\hat{B}$ under the natural projection $\pi:k\hat{G}\to kG$. 
We will prove in Theorem \ref{thm:main} that in this case the parameter $c$ in the description of the 
relations in \cite{erd} must be $c=0$. If $\hat{G}$ is a finite group with 
generalized quaternion Sylow $2$-subgroups such that $\hat{G}$ has no normal subgroups of odd order,
it follows by a result of Brauer and Suzuki \cite{brsu} that the center of $\hat{G}$ has order $2$. Therefore, if
$G=\hat{G}/Z(\hat{G})$ and if the principal block of $kG$ belongs to $\mathcal{B}_{2,d}$, then the
corresponding parameter $c$ must also be zero (see Corollary \ref{cor:principal}). As a special case, we obtain that 
if $q$ is an odd prime power and $2^d$ is the maximal $2$-power dividing $(q^2-1)$, 
then the parameter of the principal block of $k\,\mathrm{PGL}_2(\mathbb{F}_q)$, which belongs to $\mathcal{B}_{2,d}$,
must also be $c=0$ (see Corollary \ref{cor:main}).

The main idea for the proof of Theorem \ref{thm:main} is as follows. 
Let $\hat{\Lambda}$ (resp. $\Lambda$) be the basic
algebra of $\hat{B}$ (resp.  $B$). Since $B$ belongs to $\mathcal{B}_{2,d}$, it follows that
there are precisely two isomorphism classes of simple $\hat{B}$-modules and their representatives 
are given by the inflations of simple $B$-modules. 
In particular, if $B$ is the principal block of $kG$, then $\hat{B}$ is the principal block of $k\hat{G}$.
By \cite[p. 303--304]{erd}, we know the quiver and relations
of $\hat{\Lambda}$ up to a parameter $\hat{c}\in k$. Since there is a surjective $k$-algebra homomorphism
$\pi_B:\hat{B}\to B$, there is a surjective $k$-algebra homomorphism $\pi_\Lambda:
\hat{\Lambda}\to \Lambda$ which is compatible with $\pi_B$ when we view $\hat{\Lambda}$
(resp. $\Lambda$) as a (non-unitary) subalgebra of $\hat{B}$ (resp. $B$). We then determine possible generators
of the kernel $\mathrm{Ker}(\pi_\Lambda)$. To find these generators, we use that
all isomorphism classes of indecomposable $\Lambda$-modules are known, since $\Lambda$ modulo
its socle is a special biserial algebra by \cite[Thm. VI.10.1]{erd}. This knowledge gives us enough control on
$\mathrm{Ker}(\pi_\Lambda)$ to be able to deduce from the relations for $\hat{\Lambda}$ and
from the generators of $\mathrm{Ker}(\pi_\Lambda)$ that the constant $c$ occurring in the relations for 
$\Lambda\cong \hat{\Lambda}/\mathrm{Ker}(\pi_\Lambda)$ in \cite{erd} must be zero.
The surprising result is that although much less is known about $2$-modular blocks with generalized 
quaternion defect groups than for $2$-modular blocks with dihedral defect groups, the 
relationship we establish between the basic algebras $\hat{\Lambda}$ and $\Lambda$
is enough to determine the parameter $c$ occurring in the description of $\Lambda$.

We finish this introduction with one application of our results.
Let $B\in \mathcal{B}_{2,d}$ be a block of $kG$ for some finite group $G$ 
and let $V$ be 
a finitely generated $kG$-module belonging to $B$ with stable endomorphism ring $k$.
If there exists a central extension $\hat{G}$ of $G$ by a group of order $2$  and a block $\hat{B}$
of $k\hat{G}$ with generalized quaternion defect groups as above,
then the universal deformation ring 
$R(G,V)$ is isomorphic to a subquotient ring of the group ring of a defect group $D$ of $B$
over the ring $W(k)$ of infinite Witt vectors over $k$ (see \cite{def2simple}).
In fact, the conclusion about the structure of $R(G,V)$
for all $V$ belonging to an arbitrary block $B$ in $\mathcal{B}_{2,d}$ with stable 
endomorphism ring $k$ is equivalent to
the statement that the parameter $c$ associated to each block $B$ in $\mathcal{B}_{2,d}$ 
is equal to $0$.
This was our original motivation in studying these parameters.  
In \cite[Question 1.1]{cyclic} we asked whether the conclusion 
about the structure of $R(G,V)$ above
holds for all $V$ having stable endomorphism
ring $k$, without any condition on the block to which $V$ belongs.   We believe that further study
of this question will lead either to a counterexample or to new results about the structure of
blocks of group rings of finite groups.

\medskip

For background on group algebras and blocks, we refer the reader to \cite[Chaps. 1, 2 and 7]{CR}.
For background on finite dimensional algebras, and in 
particular on algebras given by quivers and relations, we refer the reader to \cite[Chaps. I-III]{ars}.

\section{Dihedral blocks with two simple modules}

\label{s:principal}

Throughout this section we make the following assumptions.

\begin{hypo}
\label{hyp:setup}
Let $k$ be an algebraically closed field of characteristic $2$, and let $d\ge 3$ be a fixed integer.
Suppose $G$ is a finite group and $B$ is a block of $kG$ with dihedral defect groups of order $2^d$ 
such that there are precisely two isomorphism classes of simple $B$-modules. 
\end{hypo}

Under these hypotheses,
it follows from \cite[Chaps. VI, IX and p. 294--295]{erd} that there exists $c\in\{0,1\}$ 
and $i\in\{1,2\}$ such that the basic algebra of $B$ is isomorphic to the symmetric algebra
$\Lambda_{i,c}$, as defined in Figure \ref{fig:basic},
and the decomposition matrix of $B$ is as in Figure \ref{fig:decompB}.
In particular, if $q$ is an odd prime power, $2^d$ is the maximal $2$-power dividing $(q^2-1)$ 
and $B$ is the principal block of $k\,\mathrm{PGL}_2(\mathbb{F}_q)$, then there exists $c\in\{0,1\}$ such that
$B$ is Morita equivalent to $\Lambda_{1,c}$ $($resp. $\Lambda_{2,c}$$)$ if $q\equiv 1\mod 4$ 
$($resp. $q\equiv 3\mod 4$$)$.
\begin{figure}[ht] \hrule \caption{\label{fig:basic} The basic algebras 
$\Lambda_{1,c}=k\,Q_1/I_{1,c}$ and $\Lambda_{2,c}=k\,Q_2/I_{2,c}$.}
$$\begin{array}{rl}
\raisebox{-2ex}{$Q_1\quad=$}&\xymatrix @R=-.2pc {
0&1\\
\ar@(ul,dl)_{\alpha} \bullet \ar@<.8ex>[r]^{\beta} &\bullet\ar@<.9ex>[l]^{\gamma}}
\quad\raisebox{-2ex}{and}\\[5ex]
I_{1,c}\quad=&\langle \beta\gamma, \alpha^2-c\,(\gamma\beta\alpha)^{2^{d-2}},
(\gamma\beta\alpha)^{2^{d-2}}-(\alpha\gamma\beta)^{2^{d-2}}\rangle,\\[3ex]
\raisebox{-2ex}{$Q_2\quad=$}&\xymatrix @R=-.2pc {
0&1\\
\ar@(ul,dl)_{\alpha} \bullet \ar@<.8ex>[r]^{\beta} &\bullet\ar@<.9ex>[l]^{\gamma}
\ar@(ur,dr)^{\eta}}\quad\raisebox{-2ex}{and}\\[5ex]
I_{2,c}\quad=& \langle \eta\beta,\gamma\eta,\beta\gamma, \alpha^2-c\,\gamma\beta\alpha,
\gamma\beta\alpha-\alpha\gamma\beta,\eta^{2^{d-2}}-\beta\alpha\gamma\rangle.
\end{array}$$
\vspace{2ex}
\hrule
\end{figure}
\begin{figure}[ht] \hrule \caption{\label{fig:decompB} The decomposition matrix for a block $B$
of $kG$ that is Morita equivalent to $\Lambda_{1,c}$ (resp. $\Lambda_{2,c}$),
where $(*)$ means that the last row is repeated $(2^{d-2}-1)$ times.}
$$\begin{array}{c@{}c}
\left[\begin{array}{cc}1&0\\1&0\\1&1\\1&1\\2&1\end{array}\right]&
\begin{array}{c}\\ \\ \\ \\ (*)\end{array}\end{array}
\qquad\qquad \left(\mbox{resp.}\quad
\begin{array}{c@{}c}
\left[\begin{array}{cc}1&0\\1&0\\1&1\\1&1\\0&1\end{array}\right]
\begin{array}{c}\\ \\ \\ \\ (*)\end{array}\end{array}\right)$$
\vspace{2ex}
\hrule
\end{figure}

Let $i\in\{1,2\}$. The projective indecomposable $\Lambda_{i,c}$-module corresponding to the vertex 
$0$ (resp. $1$) of  the quiver $Q_i$ is generated by all paths in $kQ_i$ modulo $I_{i,c}$
that start at the vertex $0$ (resp. $1$). Since $\Lambda_{i,c}$ is a symmetric algebra,
it follows that every projective indecomposable $\Lambda_{i,c}$-module has a simple socle
(see for example \cite[Prop. 9.12]{CR}). 
An elementary combinatorial argument using the generators of the ideal $I_{1,c}$ shows 
that the element $(\gamma\beta\alpha)^{2^{d-2}}$
modulo $I_{1,c}$ is not zero and lies in the socle of $\Lambda_{1,c}$. Similarly,  
$(\beta\alpha\gamma)^{2^{d-2}}$ modulo $I_{1,c}$ is not zero and
lies in the socle of $\Lambda_{1,c}$. It follows that the socle of the projective
indecomposable  $\Lambda_{1,c}$-module corresponding to the vertex $0$ (resp. $1$) of 
the quiver $Q_1$ is generated by $(\gamma\beta\alpha)^{2^{d-2}}=(\alpha\gamma\beta)^{2^{d-2}}$
modulo $I_{1,c}$ (resp.  $(\beta\alpha\gamma)^{2^{d-2}}$ modulo $I_{1,c}$). 
Thus the radical series length of the projective 
indecomposable $\Lambda_{1,c}$-modules is at least $3\cdot 2^{d-2}+1$.  
Using the generators of the ideal $I_{1,c}$, it follows that all paths of length at least $3\cdot 2^{d-2}+1$
are zero modulo $I_{1,c}$. Hence the radical series length of the projective 
indecomposable $\Lambda_{1,c}$-modules is exactly $3\cdot 2^{d-2}+1$.  
Analyzing the projective indecomposable $\Lambda_{2,c}$-module corresponding to 
the vertex $0$ of the quiver $Q_2$, one similarly sees that its socle is generated by  
$\gamma\beta\alpha=\alpha\gamma\beta$ modulo $I_{2,c}$ and its radical series length is $4$. 
For the projective indecomposable $\Lambda_{2,c}$-module corresponding to the vertex $1$ of the 
quiver $Q_2$, one has that its socle is generated by $\eta^{2^{d-2}}=\beta\alpha\gamma$ modulo 
$I_{2,c}$ and its radical series length is $4\;$ (resp. $2^{d-2}+1$) if $d=3$ (resp. if $d>3$).

It follows from the definition of string algebras in \cite[Sect. 3]{buri} and from the above description of the
socle of $\Lambda_{i,c}$ for $i\in\{1,2\}$ that $\Lambda_{i,c}/\mathrm{soc}(\Lambda_{i,c})$ is a 
string algebra.
Therefore, one can see as in \cite[Sect. I.8.11]{erd} that
the isomorphism classes of all non-projective indecomposable 
$\Lambda_{i,c}$-modules are given by string and band modules as defined in \cite[Sect. 3]{buri}.

If $i\in\{1,2\}$ and $B$ is a block of $kG$ that is Morita equivalent to $\Lambda_{i,c}$,
we can use the decomposition matrix of $B$ in Figure \ref{fig:decompB} to compute the
composition series length of the projective indecomposable $B$-modules 
(see \cite[Thm. 18.26]{CR}). It follows that if $i=1$, then
the projective indecomposable $B$-module corresponding to the vertex $0$ (resp. $1$)
of the quiver $Q_1$ has $6+6\cdot (2^{d-2}-1)=6\cdot 2^{d-2}\;$ (resp. $4+3\cdot (2^{d-2}-1)
=1+ 3\cdot 2^{d-2}$) composition factors.
If $i=2$, then the projective indecomposable $B$-module corresponding to the vertex $0$ (resp. $1$)
of the quiver $Q_2$ has $6\;$ (resp. $4+(2^{d-2}-1)=3+2^{d-2}$) composition factors.

\begin{thm}
\label{thm:main}
Assume Hypothesis $\ref{hyp:setup}$. Suppose there exists a central extension $\hat{G}$ of $G$ 
by a group of order $2$ together with a block $\hat{B}$ of $k\hat{G}$ with generalized quaternion 
defect groups such that $B$ is contained in the image of $\hat{B}$ under the natural projection 
$\pi:k\hat{G}\to kG$. Then $B$ is Morita equivalent to either $\Lambda_{1,0}$  or $\Lambda_{2,0}$, i.e. 
the parameter $c$ is $c=0$.
\end{thm}

\begin{proof}
Let $z\in\hat{G}$ be a central element of order $2$ such that $G$ is isomorphic to $\hat{G}/\langle z\rangle$.
In the following, we identify $G=\hat{G}/\langle z\rangle$. Let $\pi:k\hat{G}\to kG$ be the natural projection.
Since by assumption, $B$ is contained in the image of $\hat{B}$ under $\pi$, there is a simple $kG$-module $S$ 
belonging to $B$ whose inflation via $\pi$ belongs to $\hat{B}$. Let $T$ be an arbitrary simple $k\hat{G}$-module. 
Since $z$ has order $2$, it acts trivially on $T$. Thus $T$ is inflated 
via $\pi$ from a simple $kG$-module which we also denote by $T$. 
By  \cite[Thm. 56.12]{CR}, $T$ belongs to $B$ (resp. $\hat{B}$) if and only if 
there is a sequence of simple $kG$-modules (resp. simple $k\hat{G}$-modules)
$E_1,E_2,\ldots,E_n$ such that $E_1=S$, $E_n=T$, and for $1\le j<n$, 
$E_j$ and $E_{j+1}$ are equal or there is a non-split $G$-extension (resp. a non-split
$\hat{G}$-extension)
of one of them by the other. If $E$ and $E'$ are simple $kG$-modules whose inflations from $G$ to
$\hat{G}$ are also denoted by $E$ and $E'$, then
the Lyndon-Hochschild-Serre spectral sequence gives an exact sequence
$$0\to \mathrm{Ext}^1_{kG}(E,E')\to \mathrm{Ext}^1_{k\hat{G}}(E,E') 
\to \mathrm{Hom}_{kG}(E,E').$$
This implies immediately that $\mathrm{Ext}^1_{k\hat{G}}(E,E')$ and $\mathrm{Ext}^1_{kG}(E,E')$ have 
the same $k$-dimension for each choice of non-isomorphic $E$ and $E'$. 
Therefore, the above characterization of the simple modules which belong to $B$ (resp.
$\hat{B}$) implies that 
$\hat{B}$ has the same number of isomorphism classes of simple modules as $B$. 
We can thus identify the simple $\hat{B}$-modules with the simple $B$-modules. 
Therefore, the restriction of the natural projection $\pi$ to $\hat{B}$ gives a surjective $k$-algebra homomorphism
\begin{equation}
\label{eq:piB}
\pi_B:\hat{B}\to B.
\end{equation}
Moreover, since $\mathrm{Ker}(\pi)=(1+z)k\hat{G}$, it follows that $\mathrm{Ker}(\pi_B)=(1+z)\hat{B}$.
Thus if $\hat{P}$ is a projective indecomposable $\hat{B}$-module and $P=B\otimes_{\hat{B}}\hat{P}$
is the corresponding projective indecomposable $B$-module, 
then we have a short exact sequence of $\hat{B}$-modules
\begin{equation}
\label{eq:sesproj}
0\to (1+z)\hat{P} \to \hat{P}\to P \to 0.
\end{equation}
Since the map $f:(1+z)\hat{P}\to B\otimes_{\hat{B}} \hat{P}=P$, defined by $f((1+z)x)=1\otimes x$ for 
all $x\in \hat{P}$, is a $\hat{B}$-module isomorphism, $\hat{P}$ is a non-trivial $\hat{B}$-extension of
$P$ by itself. In particular,
$\hat{P}$ has twice as many composition factors as $P$.

Considering the possible decomposition matrices for blocks with generalized quaternion
defect groups and two isomorphism classes of simple modules as given in \cite[p. 303--304]{erd}
and using that the projective indecomposable $\hat{B}$-modules must have twice as
many composition factors as the corresponding projective indecomposable $B$-modules,
we conclude the following:
If $B$ is Morita equivalent to $\Lambda_{1,c}$ (resp.  $\Lambda_{2,c}$ ), then there
exists a constant $\hat{c}\in k$ such that the basic algebra of $\hat{B}$ is isomorphic to
$\hat{\Lambda}_{1,\hat{c}}$ (resp. $\hat{\Lambda}_{2,\hat{c}}$), as defined in Figure 
\ref{fig:basichat}, and the decomposition matrix of $\hat{B}$ is as in Figure \ref{fig:decomphatB}.
\begin{figure}[ht] \hrule \caption{\label{fig:basichat} The basic algebras 
$\hat{\Lambda}_{1,\hat{c}} =k\,Q_1/\hat{I}_{1,\hat{c}}$ and
$\hat{\Lambda}_{2,\hat{c}}=k\,Q_2/\hat{I}_{2,\hat{c}}$, where $Q_1$ and $Q_2$
are as in Figure \ref{fig:basic}.}
$$\begin{array}{rl} 
\hat{I}_{1,\hat{c}}\quad = & 
\langle \gamma\beta\gamma-\alpha\gamma(\beta\alpha\gamma)^{2^{d-1}-1},
\,\beta\gamma\beta-\beta\alpha(\gamma\beta\alpha)^{2^{d-1}-1},\\[1ex]
& \alpha^2-\gamma\beta(\alpha\gamma\beta)^{2^{d-1}-1}-\hat{c}\,(\alpha\gamma\beta)^{2^{d-1}},
\,\beta\alpha^2\rangle,\\[3ex] 
\hat{I}_{2,\hat{c}}\quad= & \langle \eta\beta - \beta\alpha(\gamma\beta\alpha),\,
\gamma\eta - \alpha\gamma(\beta\alpha\gamma), \,\beta\gamma - \eta^{2^{d-1}-1}, \\[1ex]
& \alpha^2-\gamma\beta(\alpha\gamma\beta)-\hat{c}\,(\alpha\gamma\beta)^2,\,
\beta\alpha^2\rangle.
\end{array}$$
\vspace{2ex}
\hrule
\end{figure}
\begin{figure}[ht] \hrule \caption{\label{fig:decomphatB} The decomposition matrix for a block $\hat{B}$
of $k\hat{G}$ that is Morita equivalent to $\hat{\Lambda}_{1,\hat{c}}$ (resp. 
$\hat{\Lambda}_{2,\hat{c}}$), where $(**)$ means that the last row is repeated $(2^{d-1}-1)$ times.}
$$\begin{array}{c@{}c}
\left[\begin{array}{cc}1&0\\1&0\\1&1\\1&1\\0&1\\2&1\end{array}\right]&
\begin{array}{c}\\ \\ \\ \\ \\ (**)\end{array}\end{array}
\qquad\qquad \left(\mbox{resp.}\quad
\begin{array}{c@{}c}
\left[\begin{array}{cc}1&0\\1&0\\1&1\\1&1\\2&1\\0&1\end{array}\right]
\begin{array}{c}\\ \\ \\ \\ \\ (**)\end{array}\end{array}\right)$$
\vspace{2ex}
\hrule
\end{figure}

As we did for $B$, we can use the decomposition matrix of $\hat{B}$ in Figure 
\ref{fig:decomphatB} to compute the composition series length of the projective 
indecomposable $\hat{B}$-modules.  If $\hat{B}$ is Morita equivalent to 
$\hat{\Lambda}_{1,\hat{c}}$, then the projective indecomposable $\hat{B}$-module 
corresponding to the vertex $0$ (resp. $1$) of the quiver $Q_1$ has 
$6+6\cdot (2^{d-1}-1)=6\cdot 2^{d-1}\;$ (resp. $5+3\cdot (2^{d-1}-1)=2+ 3\cdot 2^{d-1}$) 
composition factors.
If $\hat{B}$ is Morita equivalent to $\hat{\Lambda}_{2,\hat{c}}$, then the projective 
indecomposable $\hat{B}$-module corresponding to the vertex $0$ (resp. $1$)
of the quiver $Q_2$ has $12\;$ (resp. $7+(2^{d-1}-1)=6+2^{d-1}$) composition factors.

Let $E$ and $F$ be non-isomorphic simple $B$-modules.
Then the inflations of $E$ and $F$ to $\hat{B}$ are representatives of the isomorphism
classes of simple $\hat{B}$-modules, which we again denote by $E$ and $F$.
Let $\hat{e}$ and $\hat{f}$ be two orthogonal primitive idempotents in $\hat{B}$ such that
$\hat{B}\hat{e}$ (resp. $\hat{B}\hat{f}$) is a projective $\hat{B}$-module cover of
$E$ (resp. $F$).  Then $e=\pi_B(\hat{e})$ and $f=\pi_B(\hat{f})$ are
orthogonal primitive idempotents in $B$, 
and $Be$ (resp. $Bf$) is a projective $B$-module
cover of $E$ (resp. $F$). Let $\hat{\varepsilon}=\hat{e}+\hat{f}$ and $\varepsilon=e+f$.
Then $\pi_B$ from $(\ref{eq:piB})$ restricts to a surjective $k$-algebra homomorphism
\begin{equation}
\label{eq:pie}
\pi_\varepsilon: \hat{\varepsilon}\hat{B}\hat{\varepsilon} \to \varepsilon B \varepsilon.
\end{equation}
Let $i\in\{1,2\}$, $\hat{c}\in k$ and $c\in \{0,1\}$ be such that there are $k$-algebra isomorphisms
$\hat{f}_{\Lambda}:\hat{\varepsilon}\hat{B}\hat{\varepsilon} \to \hat{\Lambda}_{i,\hat{c}}$
and $f_\Lambda: \varepsilon B \varepsilon \to \Lambda_{i,c}$. Then
there exists a surjective $k$-algebra homomorphism
\begin{equation}
\label{eq:pil}
\pi_\Lambda: \hat{\Lambda}_{i,\hat{c}} \to  \Lambda_{i,c}
\end{equation}
such that the 
diagram
in Figure \ref{fig:diagram} commutes.
\begin{figure}[ht] \hrule \caption{\label{fig:diagram} The commutative diagram relating
$\hat{B}$, $\hat{\Lambda}_{i,\hat{c}}$, $B$ and $\Lambda_{i,c}$.}
$$\xymatrix  {
\hat{B} \ar[rr]^{\pi_B} && B\\
\hat{\varepsilon}\hat{B}\hat{\varepsilon} \ar[rr]^{\pi_\varepsilon}\ar@{^(->}[u] && \varepsilon B \varepsilon\ar@{^(->}[u]\\
\hat{\Lambda}_{i,\hat{c}} \ar[rr]^{\pi_\Lambda} \ar[u]^{\hat{f}_{\Lambda}}\ar[rd]_{\mathrm{nat.proj.}}&&
\Lambda_{i,c}\ar[u]_{f_\Lambda}\\
&\hat{\Lambda}_{i,\hat{c}} /\mathrm{Ker}(\pi_\Lambda)\ar[ur]_{\overline{\pi}_\Lambda}\\
&\Lambda'_{i,c}\ar@{=}[u]&
}$$
\vspace{2ex}
\hrule
\end{figure}
Here $\mathrm{nat.proj.}$ stands for the natural projection of $\hat{\Lambda}_{i,\hat{c}}$
onto $\hat{\Lambda}_{i,\hat{c}}/\mathrm{Ker}(\pi_\Lambda)$
and $\overline{\pi}_\Lambda$
is the $k$-algebra isomorphism induced by $\pi_\Lambda$.
As in Figure \ref{fig:diagram}, define 
$\Lambda'_{i,c}=\hat{\Lambda}_{i,\hat{c}}/\mathrm{Ker}(\pi_\Lambda)$.
We introduce this extra notation, since in the following we will need to distinguish between
the isomorphic $k$-algebras $\Lambda_{i,c}$ and $\Lambda'_{i,c}$.

We next show that the kernel of $\pi_\Lambda$ is contained in the square of the radical of 
$\hat{\Lambda}_{i,\hat{c}}$. This basically follows since $\hat{\Lambda}_{i,\hat{c}}$
and $\Lambda_{i,c}$ have the same quiver $Q_i$. However, for the convenience of the reader, 
we now give a more detailed explanation.
Since $\pi_\Lambda$ is surjective, $\pi_\Lambda(\mathrm{rad}^n(\hat{\Lambda}_{i,\hat{c}}))
\subseteq \mathrm{rad}^n(\Lambda_{i,c})$ for all $n\ge 1$ by \cite[Prop. 5.6]{CR}.
Thus $\pi_\Lambda$ induces a surjective $k$-algebra homomorphism 
$\pi_{\Lambda,n}:\hat{\Lambda}_{i,\hat{c}}/\mathrm{rad}^n(\hat{\Lambda}_{i,\hat{c}})\to
 \Lambda_{i,c}/\mathrm{rad}^n(\Lambda_{i,c})$ for all $n\ge 1$. 
Let $kQ_i^+$ be the ideal of $kQ_i$ generated by all arrows.
By \cite[Prop. III.1.6]{ars}, $\mathrm{rad}(\hat{\Lambda}_{i,\hat{c}})$ (resp.  
$\mathrm{rad}(\Lambda_{i,c})$) is the image of $kQ_i^+$ in 
$\hat{\Lambda}_{i,\hat{c}}$ (resp.  $\Lambda_{i,c}$).
Because $\hat{I}_{i,\hat{c}}$ (resp. $I_{i,c}$) is contained in $(kQ_i^+)^2$,
it follows that for $n\in\{1,2\}$ there is an isomorphism between $kQ_i/(kQ_i^+)^n$ and
$\hat{\Lambda}_{i,\hat{c}}/\mathrm{rad}^n(\hat{\Lambda}_{i,\hat{c}})$ (resp.
$\Lambda_{i,c}/\mathrm{rad}^n(\Lambda_{i,c})$). Hence $\pi_{\Lambda,n}$ is a $k$-algebra 
isomorphism for $n\in\{1,2\}$, which implies $\mathrm{Ker}(\pi_\Lambda)\subset  
\mathrm{rad}^2(\hat{\Lambda}_{i,\hat{c}})$. In particular, \cite[Prop. III.1.6]{ars} implies that
the radical of $\Lambda'_{i,c}$ is the image of $kQ_i^+$ in $\Lambda'_{i,c}$. 

Even though $\pi_B$ in $(\ref{eq:piB})$ is explicit and natural and the domain and range of
$\pi_\Lambda$ in $(\ref{eq:pil})$ are explicitly given, we do not have an explicit description of 
$\pi_\Lambda$, owing to the fact that Morita equivalences, given by $f_\Lambda$ and 
$\hat{f}_\Lambda$, were used to define $\pi_\Lambda$ in Figure \ref{fig:diagram}. 
Therefore, we need to show that 
$\pi_\Lambda$ matches up the simple modules corresponding to the vertices of $Q_i$.
Let $S_0$ (resp. $S_1$) be the simple $\Lambda'_{i,c}$-module corresponding
to the vertex $0$ (resp. $1$) in $Q_i$. Then the inflations of $S_0$ and $S_1$ to
$\hat{\Lambda}_{i,\hat{c}}$ are the simple $\hat{\Lambda}_{i,\hat{c}}$-modules,
which also correspond to the vertices $0$ and $1$ of $Q_i$, respectively.
Let $T_0$ (resp. $T_1$) be the simple  $\Lambda_{i,c}$-modules corresponding to the
vertex $0$ (resp. $1$) in $Q_i$. 
We computed the composition series lengths of the projective indecomposable 
$\Lambda_{i,c}$-modules $P(T_0)$ and $P(T_1)$ 
in the paragraph before the statement of Theorem \ref{thm:main},
and we computed the composition series lengths of the projective 
indecomposable $\hat{\Lambda}_{i,\hat{c}}$-modules $\hat{P}(S_0)$ and $\hat{P}(S_1)$ 
in the paragraph after defining $\hat{\Lambda}_{1,\hat{c}}$ and $\hat{\Lambda}_{2,\hat{c}}$.
Comparing these composition series lengths,
we see that $\hat{P}(S_u)$ has twice as many composition factors as
$P(T_v)$ if and only if $u=v$. Thus the Morita equivalence between $\Lambda'_{i,c}$
and $\Lambda_{i,c}$ induced by the isomorphism $\overline{\pi}_\Lambda$ must send
$S_u$ to $T_u$ for $u\in\{0,1\}$.

To complete the proof of Theorem \ref{thm:main}, 
we will make use of $\hat{\Lambda}_{i,\hat{c}}$-modules corresponding to 
certain paths in $kQ_i$. Let $w=\zeta_n\cdots\zeta_2\zeta_1$ be a path of length $n\ge 1$ in $kQ_i$ 
whose image modulo $\hat{I}_{i,\hat{c}}$ does not 
lie in $\mathrm{soc}_2(\hat{\Lambda}_{i,\hat{c}})$. 
For $1\le j\le n$, let $v_j$ be the end vertex of $\zeta_j$, and let $v_0$  be the starting
vertex of $\zeta_1$. 
Define a $kQ_i$-module $M_w$ of
$k$-dimension $n+1$ with respect to a given $k$-basis $\{b_0,\ldots,b_n\}$ as follows. 
Let $0\le j\le n$.
If $v$ is a vertex in $Q_i$, define $v b_j=b_j$ if $v=v_j$, and $v b_j=0$ otherwise. If
$\zeta$ is an arrow in $Q_i$, define $\zeta b_j=b_{j+1}$ if $\zeta=\zeta_{j+1}$ and $j\le n-1$, 
otherwise define $\zeta b_j=0$. 
By our assumption on $w$, the ideal $\hat{I}_{i,\hat{c}}$ of $kQ_i$ acts as zero on $M_w$.
Hence $M_w$ defines a $\hat{\Lambda}_{i,\hat{c}}$-module, which we also denote by $M_w$.
Moreover, $M_w$ is a uniserial $\hat{\Lambda}_{i,\hat{c}}$-module with descending
composition factors $S_{v_0},S_{v_1},\ldots,S_{v_n}$.

Suppose first that $i=1$. 
As seen in the second paragraph after Hypothesis \ref{hyp:setup}, the radical series
lengths of the projective $\Lambda_{1,c}$-module covers of $T_0$ and $T_1$ 
are both $3\cdot 2^{d-2}+1$.
Using the Morita equivalence between $\Lambda'_{1,c}$ and 
$\Lambda_{1,c}$ induced by $\overline{\pi}_\Lambda$, it follows that
the radical series lengths of the projective $\Lambda'_{1,c}$-module covers of $S_0$ and $S_1$
are also both $3\cdot 2^{d-2}+1$.
Since both the radical of $\hat{\Lambda}_{1,\hat{c}}$ and the radical of $\Lambda'_{1,c}$ are generated by arrows, all paths in $kQ_1$ of length greater than or equal to $3\cdot 2^{d-2}+1$ modulo 
$\hat{I}_{1,\hat{c}}$ must lie in $\mathrm{Ker}(\pi_\Lambda)$. In particular,
$$\gamma\beta(\alpha\gamma\beta)^{2^{d-1}-1}+\hat{c}\,(\alpha\gamma\beta)^{2^{d-1}}
\quad \mbox{modulo $\hat{I}_{1,\hat{c}}$}$$
lies in $\mathrm{Ker}(\pi_\Lambda)$. Thus it follows from the description of $\hat{I}_{1,\hat{c}}$ 
in Figure \ref{fig:basichat} that $\alpha^2$ modulo $\hat{I}_{1,\hat{c}}$ lies in 
$\mathrm{Ker}(\pi_\Lambda)$. Using similar arguments, we see that
\begin{equation}
\label{eq:ohyeah1}
\alpha^2,\beta\gamma\beta, \gamma\beta\gamma,
(\alpha\gamma\beta)^{2^{d-2}}\alpha,(\gamma\beta\alpha)^{2^{d-2}}\gamma,
(\beta\alpha\gamma)^{2^{d-2}}\beta \quad \mbox{modulo $\hat{I}_{1,\hat{c}}$}
\end{equation}
all lie in $\mathrm{Ker}(\pi_\Lambda)$. 
Taking the path $\beta\gamma$ in $kQ_1$, we can define the uniserial 
$\hat{\Lambda}_{1,\hat{c}}$-module $M_{\beta\gamma}$ as above
with descending composition factors $S_1,S_0,S_1$.
Using that $\Lambda_{1,c}/\mathrm{soc}(\Lambda_{1,c})$ is a string algebra, we see that
there is no uniserial $\Lambda_{1,c}$-module with descending composition factors
$T_1,T_0,T_1$. Because of the Morita equivalence between $\Lambda'_{1,c}$
and $\Lambda_{1,c}$ induced by $\overline{\pi}_\Lambda$, it follows 
that there is also no uniserial $\Lambda'_{1,c}$-module with descending composition 
factors $S_1,S_0,S_1$.
Therefore, there must exist an element $x\in \mathrm{Ker}(\pi_\Lambda)$ that acts non-trivially
on $M_{\beta\gamma}$. The only elements of $\hat{\Lambda}_{1,\hat{c}}$ that act
non-trivially on $M_{\beta\gamma}$ are $k$-linear combinations of the paths
$1_0$, $1_1$, $\beta$, $\gamma$ and $\beta\gamma$ modulo $\hat{I}_{1,\hat{c}}$.
Since $\mathrm{Ker}(\pi_\Lambda) \subset \mathrm{rad}^2(\hat{\Lambda}_{1,\hat{c}})$,
it follows that $x=\beta\gamma + x'$ modulo $\hat{I}_{1,\hat{c}}$ 
where $x'$ is a (possibly zero) $k$-linear combination of paths 
$z$ of length at least 2 that are different from $\beta\gamma$.
Multiplying $x$ on both sides with $1_1$ modulo $\hat{I}_{1,\hat{c}}$, we
may assume that all the paths $z$ coming up in the description of $x'$ both start and end at
the vertex $1$. Moreover, we may omit all paths $z$ that modulo $\hat{I}_{1,\hat{c}}$ lie
in $\mathrm{Ker}(\pi_\Lambda)$ according to $(\ref{eq:ohyeah1})$. Thus we may assume that
$x'$ is a $k$-linear combination of $(\beta\alpha\gamma)^j$ for  
$1\le j\le 2^{d-2}$. It follows that there exist constants $a_1,\ldots,a_{2^{d-2}}$ in $k$  such that
\begin{equation}
\label{eq:desperate1}
x\;=\;\beta\gamma \;+\; \sum_{j=1}^{2^{d-2}} a_j \,(\beta\alpha\gamma)^j\quad
\mbox{modulo $\hat{I}_{1,\hat{c}}$}
\end{equation}
lies in $\mathrm{Ker}(\pi_\Lambda)$.
Suppose not all $a_j$ are zero, and let $1\le j_0\le 2^{d-2}$ be minimal with 
$a_{j_0}\neq 0$. If $j_0< 2^{d-2}$, we can multiply $x$ on the left with
$\gamma(\beta\alpha\gamma)^{2^{d-2}-j_0-1}$ modulo $\hat{I}_{1,\hat{c}}$ (resp. on the right with 
$(\beta\alpha\gamma)^{2^{d-2}-j_0-1}\beta$ modulo $\hat{I}_{1,\hat{c}}$)
and use $(\ref{eq:ohyeah1})$ to see that
$$\gamma(\beta\alpha\gamma)^{2^{d-2}-1}\quad\mbox{and}\quad
(\beta\alpha\gamma)^{2^{d-2}-1}\beta \qquad \mbox{modulo $\hat{I}_{1,\hat{c}}$}$$
lie in $\mathrm{Ker}(\pi_\Lambda)$. But this implies that all paths of length $3\cdot2^{d-2}$ 
modulo $\hat{I}_{1,\hat{c}}$ lie in
$\mathrm{Ker}(\pi_\Lambda)$, and
hence $\mathrm{rad}^{3\cdot2^{d-2}}(\Lambda'_{1,c})=0$, which is a contradiction. 
Thus $a_j=0$ for $1\le j< 2^{d-2}$ and we obtain 
\begin{equation}
\label{eq:nice1}
\pi_\Lambda\left(\beta\gamma-c_1 \,((\beta\alpha\gamma)^{2^{d-2}}\;
\mbox{modulo $\hat{I}_{1,\hat{c}}$}\right) = 0
\end{equation}
for some $c_1\in k$. 
Using that $\Lambda_{1,c}/\mathrm{soc}(\Lambda_{1,c})$ is a string algebra, we see that
there is no uniserial $\Lambda_{1,c}$-module of length $3\cdot 2^{d-2}+1$ with descending 
composition factors
$$(T_0,T_0,T_1,T_0,T_0,\ldots,T_1,T_0,T_0,T_1,T_0)$$
$$\mbox{(resp. $(T_0,T_1,T_0,T_0,T_1,\ldots,T_0,T_0,T_1,T_0,T_0)$)}.$$
Because of the Morita equivalence between $\Lambda'_{1,c}$
and $\Lambda_{1,c}$ induced by $\overline{\pi}_\Lambda$, we see that
there is also no uniserial $\Lambda'_{1,c}$-module of length $3\cdot 2^{d-2}+1$ with descending
composition factors
$$(S_0,S_0,S_1,S_0,S_0,\ldots,S_1,S_0,S_0,S_1,S_0)$$
$$\mbox{(resp. $(S_0,S_1,S_0,S_0,S_1,\ldots,S_0,S_0,S_1,S_0,S_0)$)}.$$
On the other hand, we can define a uniserial $\hat{\Lambda}_{1,\hat{c}}$-module
$M_w$ corresponding to the path $w=(\gamma\beta\alpha)^{2^{d-2}}$ 
(resp. $w=(\alpha\gamma\beta)^{2^{d-2}}$) in $kQ_1$ with such descending composition factors.
Thus there must exist an element $y\in \mathrm{Ker}(\pi_\Lambda)$ that acts non-trivially
on $M_w$. The only elements of $\hat{\Lambda}_{1,\hat{c}}$ that act
non-trivially on $M_w$ are $k$-linear combinations of 
$1_0$, $1_1$, $\alpha$, $\beta$, $\gamma$ and subpaths of $w$ of length at least $2$ 
modulo $\hat{I}_{1,\hat{c}}$.
Since $\mathrm{Ker}(\pi_\Lambda) \subset \mathrm{rad}^2(\hat{\Lambda}_{1,\hat{c}})$,
there exists a subpath $w'$ of $w$ of length at least $2$ such that 
$y=w' + y'$ modulo $\hat{I}_{1,\hat{c}}$ and $y'$ is a (possibly zero) $k$-linear combination of paths 
$z\neq w'$ of length at least 2.
Since $w=w_1 w' w_2$ for certain subpaths $w_1$ and $w_2$ of $w$, we can multiply $y$ on the 
left with $w_1$ modulo $\hat{I}_{1,\hat{c}}$ and on the right with $w_2$ modulo $\hat{I}_{1,\hat{c}}$
to be able to assume that $w'=w$.
Multiplying $y$ on both sides with $1_0$ modulo $\hat{I}_{1,\hat{c}}$, we
may assume that all the paths $z$ coming up in the description of $y'$ both start and end at
the vertex $0$. Moreover, we may omit all paths $z$ that modulo $\hat{I}_{1,\hat{c}}$ lie
in $\mathrm{Ker}(\pi_\Lambda)$ according to $(\ref{eq:ohyeah1})$. 
If $\tilde{w}$ is the path such that $\{w,\tilde{w}\} = \{(\gamma\beta\alpha)^{2^{d-2}},
(\alpha\gamma\beta)^{2^{d-2}}\}$, we may thus assume that
$y'$ is a $k$-linear combination of $\tilde{w}$, $(\gamma\beta\alpha)^j$ and $(\alpha\gamma\beta)^j$
for $1\le j\le 2^{d-2}-1$. It follows that there exist constants $a,a_{j,1}, a_{j,2}$ in $k$ for
$1\le j\le 2^{d-2}-1$ such that
\begin{equation}
\label{eq:desperate2}
y\;=\;w \;+\; a\, \tilde{w} \;+\;\sum_{j=1}^{2^{d-2}-1} a_{j,1} \,(\gamma\beta\alpha)^j
\;+\; \sum_{j=1}^{2^{d-2}-1} a_{j,2}\, (\alpha\gamma\beta)^j\qquad
\mbox{modulo $\hat{I}_{1,\hat{c}}$}
\end{equation}
lies in $\mathrm{Ker}(\pi_\Lambda)$.
Arguing in a similar way to $(\ref{eq:nice1})$, we see that all $a_{j,1}=0=a_{j,2}$ for $1\le j\le 2^{d-2}-1$.
We obtain
\begin{eqnarray}
\label{eq:nice2}
\pi_\Lambda\left((\gamma\beta\alpha)^{2^{d-2}})-c_2\, (\alpha\gamma\beta)^{2^{d-2}}\;
\mbox{modulo $\hat{I}_{1,\hat{c}}$}\right) &=& 0,\\
\label{eq:nice3}
\pi_\Lambda\left((\alpha\gamma\beta)^{2^{d-2}})-c_3\, (\gamma\beta\alpha)^{2^{d-2}}\;
\mbox{modulo $\hat{I}_{1,\hat{c}}$}\right) &=& 0
\end{eqnarray}
for certain $c_2,c_3\in k$. 
If one of $c_2,c_3$ were zero, we could also choose the other one to be zero. This would imply that
the radical series length of the projective $\Lambda'_{1,c}$-module cover of $S_0$ is at most
$3\cdot2^{d-2}$, which is a contradiction. 
Thus $c_2,c_3$ have to be non-zero. Therefore,
there exist certain $c_1\in k$, $c_2\in k^*$ such that
$\mathrm{Ker}(\pi_\Lambda)$ contains the image $\hat{J}_{c_1,c_2}$ in $\hat{\Lambda}_{1,\hat{c}}$ 
of the ideal 
$$J_{c_1,c_2}=\langle \alpha^2, \beta\gamma-c_1(\beta\alpha\gamma)^{2^{d-2}},
(\gamma\beta\alpha)^{2^{d-2}}-c_2(\alpha\gamma\beta)^{2^{d-2}}\rangle \;\subset \; kQ_1.$$
We can compute the $k$-dimension of $kQ_1/I_{1,c}=\Lambda_{1,c}$ (resp. of $kQ_1/J_{c_1,c_2}
\cong \hat{\Lambda}_{1,\hat{c}}/\hat{J}_{c_1,c_2}$) by counting $k$-linearly independent paths in 
$kQ_1$ modulo $I_{1,c}$ (resp. modulo $J_{c_1,c_2}$). Since these 
$k$-dimensions are equal, it follows that $\mathrm{Ker}(\pi_\Lambda)=\hat{J}_{c_1,c_2}$ and thus
$\Lambda'_{1,c}\cong kQ_1/J_{c_1,c_2}$. 
Because $\Lambda'_{1,c}\cong kQ_1/J_{c_1,c_2}$ is symmetric, there exists a $k$-linear map
$\varphi:kQ_1/J_{c_1,c_2}\to k$ such that $\varphi(ab)=\varphi(ba)$ for all $a,b\in kQ_1/J_{c_1,c_2}$
and such that $\mathrm{Ker}(\varphi)$ does not contain any non-zero left ideal of $kQ_1/J_{c_1,c_2}$.
As in the second paragraph after Hypothesis \ref{hyp:setup}, one sees that the socle of the
projective indecomposable $kQ_1/J_{c_1,c_2}$-module corresponding to the vertex $0$ is
generated by $(\gamma\beta\alpha)^{2^{d-2}}$ modulo $J_{c_1,c_2}$. It follows that
$$\varphi\left((\gamma\beta\alpha)^{2^{d-2}}\; \mbox{modulo $J_{c_1,c_2}$}\right) =
\varphi\left((\alpha\gamma\beta)^{2^{d-2}}\; \mbox{modulo $J_{c_1,c_2}$}\right)$$
is non-zero, which implies $c_2=1$.
Hence $\Lambda'_{1,c}\cong kQ_1/J_{c_1,1}$ for a certain $c_1\in k$.
Since the $k$-algebra homomorphism
\begin{eqnarray*}
\Lambda_{1,0}=kQ_1/I_{1,0} &\longrightarrow & kQ_1/J_{c_1,1}\\
z \quad \mbox{modulo $I_{1,0}$}&\mapsto & z\quad \mbox{modulo $J_{c_1,1}$}\qquad\qquad
\mbox{ for $z\in\{1_0,1_1,\alpha,\beta\}$},\\
\gamma\quad \mbox{modulo $I_{1,0}$}&\mapsto &
\gamma - c_1\alpha\gamma(\beta\alpha\gamma)^{2^{d-2}-1}\quad \mbox{modulo $J_{c_1,1}$}
\end{eqnarray*}
is a $k$-algebra isomorphism for all $c_1\in k$, it follows that $\Lambda'_{1,c}\cong\Lambda_{1,0}$.
Therefore, the parameter $c$ must be equal to zero.

Suppose next that $i=2$. As seen in the second paragraph after Hypothesis \ref{hyp:setup}, 
the radical series length of the projective $\Lambda_{2,c}$-module cover of $T_0$ (resp. $T_1$) 
is $4$ (resp. $4$ if $d=3$ and $2^{d-2}+1$ if $d>3$). Using the Morita equivalence between
$\Lambda'_{2,c}$ and $\Lambda_{2,c}$ induced by $\overline{\pi}_\Lambda$, it follows that
the radical series length of the projective $\Lambda'_{2,c}$-module cover of $S_0$ (resp. $S_1$) 
is also $4$ (resp. $4$ if $d=3$ and $2^{d-2}+1$ if $d>3$). 
Since both the radical of $\hat{\Lambda}_{2,\hat{c}}$ and the radical of $\Lambda'_{2,c}$ are generated by arrows, all paths in $kQ_2$ starting at vertex $0$ of length greater than or
equal to $4$ modulo $\hat{I}_{2,\hat{c}}$ must lie in $\mathrm{Ker}(\pi_\Lambda)$. Also, 
if $d=3$ (resp. if $d>3$), then all paths in $kQ_2$ starting at vertex $1$ of length greater 
than or equal to $4$ (resp. $2^{d-2}+1$) modulo $\hat{I}_{2,\hat{c}}$ must lie in 
$\mathrm{Ker}(\pi_\Lambda)$. 
Thus we can argue similarly to $(\ref{eq:ohyeah1})$, using
the description of $\hat{I}_{2,\hat{c}}$ in Figure \ref{fig:basichat}, that 
\begin{eqnarray}
\label{eq:ohyeah2d3}
\alpha(\gamma\beta\alpha), \beta(\alpha\gamma\beta), 
\alpha^2, \eta\beta, \gamma\eta, \gamma(\beta\alpha\gamma),\eta^4\;
&\mbox{modulo $\hat{I}_{2,\hat{c}}$}& \mbox{ if $d=3$}\\
\label{eq:ohyeah2}
(\,\mbox{resp. }\; \alpha(\gamma\beta\alpha), \beta(\alpha\gamma\beta), 
\alpha^2, \eta\beta, \gamma\eta,\beta\gamma, \eta^{2^{d-2}+1}
&\mbox{modulo $\hat{I}_{2,\hat{c}}$}& \mbox{ if $d>3$}\; )
\end{eqnarray}
all lie in $\mathrm{Ker}(\pi_\Lambda)$. 
We first show that for $d=3$, we also have that $\beta\gamma$ and $\eta^{2^{d-2}+1}$ 
modulo $\hat{I}_{2,\hat{c}}$ lie in $\mathrm{Ker}(\pi_\Lambda)$. 
Using that $\Lambda_{2,c}/\mathrm{soc}(\Lambda_{2,c})$ is a string algebra, we see that if $d=3$,
there is no uniserial $\Lambda_{2,c}$-module  with descending composition factors $T_1,T_1,T_1$.
Because of the Morita equivalence between $\Lambda'_{2,c}$ and $\Lambda_{2,c}$ induced by 
$\overline{\pi}_\Lambda$, there is then also no uniserial $\Lambda'_{2,c}$-module with descending 
composition factors $S_1,S_1,S_1$. However, we can define the uniserial 
$\hat{\Lambda}_{2,\hat{c}}$-module $M_{\eta^2}$ corresponding to the path $\eta^2$ in $kQ_2$
with such descending composition factors.  Since $\beta\gamma = \eta^3$ modulo 
$\hat{I}_{2,\hat{c}}$ if $d=3$, we can use $(\ref{eq:ohyeah2d3})$ to argue similarly to 
$(\ref{eq:desperate1})$ that there are constants $b_1,b_2$ in $k$ such that
\begin{equation}
\label{eq:ohsodesperate}
\eta^2\;+\;b_1\,\beta\alpha\gamma\;+\;b_2\,\eta^3 \qquad
\mbox{modulo $\hat{I}_{2,\hat{c}}$}
\end{equation}
lies in $\mathrm{Ker}(\pi_\Lambda)$ if $d=3$. Multiplying $(\ref{eq:ohsodesperate})$ on the left with
$\eta$ modulo $\hat{I}_{2,\hat{c}}$ and using that $\eta\beta$ and $\eta^4$ 
modulo $\hat{I}_{2,\hat{c}}$
lie in $\mathrm{Ker}(\pi_\Lambda)$ by $(\ref{eq:ohyeah2d3})$, we see that $\eta^3$, and thus 
$\beta\gamma$, modulo $\hat{I}_{2,\hat{c}}$ lie in $\mathrm{Ker}(\pi_\Lambda)$ if $d=3$. 
Thus it follows from $(\ref{eq:ohyeah2d3})$ and $(\ref{eq:ohyeah2})$ that for all $d\ge 3$,
\begin{equation}
\label{eq:ohyeah3}
\alpha(\gamma\beta\alpha), \beta(\alpha\gamma\beta), 
\alpha^2, \eta\beta, \gamma\eta,\beta\gamma, \eta^{2^{d-2}+1}
\qquad \mbox{modulo $\hat{I}_{2,\hat{c}}$}
\end{equation}
all lie in $\mathrm{Ker}(\pi_\Lambda)$.
Using that $\Lambda_{2,c}/\mathrm{soc}(\Lambda_{2,c})$ is a string algebra, we see that
there is no uniserial $\Lambda_{2,c}$-module  with descending 
composition factors $T_1,T_0,T_0,T_1$. Because of the Morita equivalence between 
$\Lambda'_{2,c}$ and $\Lambda_{2,c}$ induced by $\overline{\pi}_\Lambda$, we see that
there is also no uniserial $\Lambda'_{2,c}$-module with descending composition factors
$S_1,S_0,S_0,S_1$. 
On the other hand, we can define the uniserial $\hat{\Lambda}_{2,\hat{c}}$-module
$M_{\beta\alpha\gamma}$ corresponding to the path $\beta\alpha\gamma$ in $kQ_2$
with such descending composition factors.
Thus we can argue similarly to $(\ref{eq:desperate1})$ and $(\ref{eq:desperate2})$,
using $(\ref{eq:ohyeah3})$, that there exist constants $a_2,\ldots,a_{2^{d-2}}$ in $k$ such that
\begin{equation}
\label{eq:desperate3}
\beta\alpha\gamma\;+\;\sum_{j=2}^{2^{d-2}} a_j \,\eta^j\qquad
\mbox{modulo $\hat{I}_{2,\hat{c}}$}
\end{equation}
lies in $\mathrm{Ker}(\pi_\Lambda)$.
Suppose that not all $a_j$ are zero, and let $2\le j_0\le 2^{d-2}$ be minimal with 
$a_{j_0}\neq 0$. If $j_0< 2^{d-2}$ then $d>3$. Since $\eta\beta$ and $\eta^{2^{d-2}+1}$
modulo $\hat{I}_{2,\hat{c}}$ lie in $\mathrm{Ker}(\pi_\Lambda)$ by $(\ref{eq:ohyeah3})$, 
we can multiply $(\ref{eq:desperate3})$
on the left with $\eta^{2^{d-2}-j_0}$ modulo $\hat{I}_{2,\hat{c}}$ 
to see that $\eta^{2^{d-2}}$ modulo $\hat{I}_{2,\hat{c}}$
lies in $\mathrm{Ker}(\pi_\Lambda)$. But this implies that all paths of length $2^{d-2}$ 
modulo $\hat{I}_{2,\hat{c}}$ lie in $\mathrm{Ker}(\pi_\Lambda)$, and
hence $\mathrm{rad}^{2^{d-2}}(\Lambda'_{2,c})=0$, which is a contradiction. 
Thus $a_j=0$ for $2\le j< 2^{d-2}$, and we obtain 
\begin{equation}
\label{eq:nice1a}
\pi_\Lambda\left(\beta\alpha\gamma-c_1 \,\eta^{2^{d-2}}\;
\mbox{modulo $\hat{I}_{2,\hat{c}}$}\right) = 0
\end{equation}
for some $c_1\in k$. Similarly, one sees that 
\begin{equation}
\label{eq:nice2a}
\pi_\Lambda\left(\eta^{2^{d-2}}-c_2 \,\beta\alpha\gamma\;\,
\mbox{modulo $\hat{I}_{2,\hat{c}}$}\right) = 0
\end{equation}
for some $c_2\in k$. 
If one of $c_1,c_2$ were zero, we could also choose the other one to be zero. This would imply that
the radical series length of the projective $\Lambda'_{2,c}$-module cover of $S_1$ is at most
$3$ (resp. $2^{d-2}$) if $d=3$ (resp. if $d>3$),
which is a contradiction.
Thus $c_1,c_2$ have to be non-zero. 
Using that $\Lambda_{2,c}/\mathrm{soc}(\Lambda_{2,c})$ is a string algebra, we see that
there is no uniserial $\Lambda_{2,c}$-module  with descending composition factors 
$$(T_0,T_0,T_1,T_0) \quad\mbox{(resp. $(T_0,T_1,T_0,T_0)$)}.$$
Because of the Morita equivalence between 
$\Lambda'_{2,c}$ and $\Lambda_{2,c}$ induced by $\overline{\pi}_\Lambda$, we see that
there is also no uniserial $\Lambda'_{2,c}$-module with descending composition factors
$$(S_0,S_0,S_1,S_0) \quad\mbox{(resp. $(S_0,S_1,S_0,S_0)$)}.$$
On the other hand, we can define a uniserial $\hat{\Lambda}_{2,\hat{c}}$-module
$M_w$ with $w=\gamma\beta\alpha$ (resp. $w=\alpha\gamma\beta$)
with such descending composition factors.
Hence we obtain in a similar way to $(\ref{eq:nice1a})$ and $(\ref{eq:nice2a})$
\begin{eqnarray}
\label{eq:nice3a}
\pi_\Lambda\left(\gamma\beta\alpha - c_3\, \alpha\gamma\beta\;\,
\mbox{modulo $\hat{I}_{2,\hat{c}}$}\right) &=& 0,\\
\label{eq:nice4a}
\pi_\Lambda\left(\alpha\gamma\beta-c_4\, \gamma\beta\alpha\;\,
\mbox{modulo $\hat{I}_{2,\hat{c}}$}\right)&=& 0
\end{eqnarray}
for certain $c_3,c_4\in k$. Moreover, since the radical series length of the projective 
$\Lambda'_{2,c}$-module cover of $S_0$ is $4$, we can argue as above that $c_3,c_4$ 
have to be non-zero. 
Therefore, there exist certain $c_1,c_3\in k^*$ such that
$\mathrm{Ker}(\pi_\Lambda)$ contains the image $\hat{J}_{c_1,c_3}$ in $\hat{\Lambda}_{2,\hat{c}}$ 
of the ideal 
$$J_{c_1,c_3}=\langle \alpha^2, \eta\beta,\gamma\eta,\beta\gamma,
\beta\alpha\gamma-c_1\eta^{2^{d-2}},\gamma\beta\alpha-c_3 \alpha\gamma\beta\rangle 
\;\subset \; kQ_2.$$
We can compute the $k$-dimension of $kQ_2/I_{2,c}=\Lambda_{2,c}$ (resp. of $kQ_2/J_{c_1,c_3}
\cong \hat{\Lambda}_{2,\hat{c}}/\hat{J}_{c_1,c_3}$) by counting $k$-linearly independent paths in 
$kQ_2$ modulo $I_{2,c}$ (resp. modulo $J_{c_1,c_3}$). Since these 
$k$-dimensions are equal, it follows that $\mathrm{Ker}(\pi_\Lambda)=\hat{J}_{c_1,c_3}$ and thus
$\Lambda'_{2,c}\cong kQ_2/J_{c_1,c_3}$. 
Because $\Lambda'_{2,c}$ is symmetric, 
we can argue in a similar way as for $\Lambda'_{1,c}$ that $c_3=1$.
Hence $\Lambda'_{2,c}\cong kQ_2/J_{c_1,1}$ for a certain $c_1\in k^*$.
Since the $k$-algebra homomorphism
\begin{eqnarray*}
\Lambda_{2,0}=kQ_2/I_{2,0} &\longrightarrow & kQ_2/J_{c_1,1}\\
z \quad \mbox{modulo $I_{2,0}$}&\mapsto & z\quad \mbox{modulo $J_{c_1,1}$}\qquad
\mbox{ for $z\in\{1_0,1_1,\alpha,\beta,\eta\}$},\\
\gamma\quad \mbox{modulo $I_{2,0}$}&\mapsto &
c_1^{-1}\,\gamma\quad \mbox{modulo $J_{c_1,1}$}
\end{eqnarray*}
is a $k$-algebra isomorphism for all $c_1\in k^*$, it follows that 
$\Lambda'_{2,c}\cong\Lambda_{2,0}$.
Therefore, the parameter $c$ must be equal to zero.
\end{proof}

\begin{cor}
\label{cor:principal}
Let $k$ be an algebraically closed field of characteristic $2$, and let $d\ge 3$ be an integer. 
Suppose $\hat{G}$ is a finite group with generalized quaternion 
Sylow $2$-subgroups of order $2^{d+1}$ such that $\hat{G}$ has no normal subgroups of
odd order. Let $Z(\hat{G})$ denote the center of $\hat{G}$ and define $G=\hat{G}/Z(\hat{G})$. 
Then $G$ has dihedral Sylow $2$-subgroups of order $2^d$. If the principal block $B$ of
$kG$ has precisely two isomorphism classes of simple modules, then $B$ is Morita equivalent 
to either $\Lambda_{1,0}$  or $\Lambda_{2,0}$.
\end{cor}

\begin{proof}
By a result of Brauer and Suzuki \cite{brsu}, the center of $\hat{G}$ is generated by an element of order $2$.
Since the principal block $B$ of $kG$ is contained in the image of the principal block $\hat{B}$ of $k\hat{G}$ under
the natural projection $\pi:k\hat{G}\to kG$, Corollary \ref{cor:principal} follows from Theorem \ref{thm:main}.
\end{proof}

\begin{cor}
\label{cor:main}
Let $k$ be an algebraically closed field of characteristic $2$.
Suppose $q$ is an odd prime power and $B$ is the principal block of 
$k\,\mathrm{PGL}_2(\mathbb{F}_q)$. Further, let $2^d\ge 8$ be the maximal $2$-power dividing
$(q^2-1)$. Then $B$ is Morita equivalent to $\Lambda_{1,0}$ 
$($resp. $\Lambda_{2,0}$$)$ if $q\equiv 1\mod 4$ $($resp. $q\equiv 3\mod 4$$)$.
\end{cor}

\begin{proof}
By a result of Dickson (see for example \cite[Hauptsatz II.8.27]{hup}), the projective special linear group
$\mathrm{PSL}_2(\mathbb{F}_{q^2})$ contains a subgroup $H_q$ that is isomorphic to
$\mathrm{PGL}_2(\mathbb{F}_q)$. Consider the short exact sequence of groups
\begin{equation}
\label{eq:yay}
1\to \{\pm I\} \to \mathrm{SL}_2(\mathbb{F}_{q^2}) \xrightarrow{\sigma} 
\mathrm{PSL}_2(\mathbb{F}_{q^2}) \to 1
\end{equation}
where $I$ is the $2\times 2$ identity matrix in $\mathrm{SL}_2(\mathbb{F}_{q^2})$
and $\sigma$ is the natural surjection.
Define $\hat{G}_q$ to be the full preimage in $\mathrm{SL}_2(\mathbb{F}_{q^2})$ of $H_q$ under
the map $\sigma$. Then $\hat{G}_q/\{\pm I\}$ is isomorphic to $\mathrm{PGL}_2(\mathbb{F}_q)$.
Moreover, the center of $\hat{G}_q$ is $\{\pm I\}$, since the center of $H_q\cong
\mathrm{PGL}_2(\mathbb{F}_q)$ is trivial. The order of $\hat{G_q}$ is twice the order of
$\mathrm{PGL}_2(\mathbb{F}_q)$, and thus equal to $2\cdot q(q^2-1)$. 
Since the order of $\mathrm{SL}_2(\mathbb{F}_{q^2})$ is equal to
$q^2 (q^2-1) (q^2+1)$, it follows that the Sylow $2$-subgroups of $\mathrm{SL}_2(\mathbb{F}_{q^2})$
and the Sylow $2$-subgroups of $\hat{G}_q$ have order $2^{d+1}$. 
Because
$\mathrm{SL}_2(\mathbb{F}_{q^2})$ has generalized quaternion Sylow $2$-subgroups
(see for example \cite[Satz II.8.10]{hup}), it follows
that the Sylow $2$-subgroups of $\hat{G}_q$ are also generalized quaternion. 
Hence Corollary \ref{cor:main} is a consequence of Corollary \ref{cor:principal}.
\end{proof}


\end{document}